\documentclass[11pt]{article}

\usepackage{arxiv}

\usepackage[utf8]{inputenc} 
\usepackage[T1]{fontenc}    
\usepackage{hyperref}       
\usepackage{url}            
\usepackage{booktabs}       
\usepackage{amsfonts}       
\usepackage{nicefrac}       
\usepackage{microtype}      
\usepackage{lipsum}
\usepackage[english]{babel}
\usepackage{amsthm}

\newtheorem{theorem}{Theorem}[section]

\newtheorem{lemma}[theorem]{Lemma}
\newtheorem{example}{Example}[section]

\title{A Classification of Autoparatopisms of Latin Cubes} 

\author{
  Vindula Kumaranayake\\
  Department of Mathematics\\
  University of Colombo\\
  Colombo 03, 00300 \\
  Sri Lanka\\
  \texttt{2015s15193@stu.cmb.ac.lk} \\
}

\begin{document}
\maketitle

\begin{abstract}
A paratopism is an action on a Latin hypercube of dimension d and order n which is an element of the wreath product $S_n \wr S_{d+1} $. A paratopism is said to be an autoparatopism if there is at least one Latin hypercube which is mapped to itself under the action of the paratopism. In this paper we classify autoparatopisms of Latin cubes given $d = 3$ and $n \in \mathbb{Z^+}$, upto the conjugacy in $S_n \wr S_4 $. In order to achieve this objective, we prove that given an autoparatopism $\sigma \in S_n \wr S_{d+1}$, all the conjugates of $\sigma$ are autoparatpisms. Also, an important condition is presented, which states that the cycle structure of the $\delta$ in the element $\sigma = (\alpha_1,\alpha_2,\alpha_3,\alpha_4;\delta)$ of $S_n \wr S_4$ determines the conjugacy. 
\end{abstract}


\section{Introduction}

\subsection{Background}

A \textit{Latin square} of order n (where n $\in \mathbb{Z^+}$) is an $n\times n$ array which contains n distinct symbols, such that each symbol occurs exactly once in each row and each column as defined in \cite{keedwelldenes} 

The word "Latin Square" was assigned to above defined $n\times n$ arrays because the Swiss mathematician Leonhard Euler used Latin characters as symbols in his mathematical papers. Korean mathematician Choi Seok-jeong was the first person to publish an example of a Latin Square in 1700, (67 years prior to Euler).

The motivation was gained from \cite{mendis2017autoparatopisms} to extend its results to cubes. Primarily it was questioned that what results regarding autoparatopisms on Latin squares are true for Latin cubes.

Applications of Latin squares, Latin Cubes and Latin hypercubes can be found in areas such as Statistics (design of experiments), Mathematics (generalizations of groups, and operations in to n-ary sense), Coding (error correcting codes) and as Mathematical puzzles.

\subsection{Basic notations and definitions}
We adopt the definitions regarding to Latin hypercubes from \cite{mckay2008census} and change them appropriately for cubes as in the following manner. Let $n \in \mathbf{Z^+}$. We denote the set \{1,2,...,n\} by [n], and $[n]\times[n]\times[n]$ by $[n]^3$. Also an element $(i_1,i_2,i_3) \in [n]^3$ is denoted by $\vec{v}$. A 3-dimensional array of $n^3$ cells is said to be a \textit{cube of order n} if each entry is assigned by an element of [n]. The entry at $\vec{v}$ of a cube C is denoted by $C(\vec{v})$ or $C(i_1,i_2,i_3)$. Let C be a cube of order n, Let $i_1,i_2,...i_{k-1},i_{k+1},...,i_3 \in [n]$ be fixed with $1\le k \le 3$. Then the set $\{(i_1,i_2,...i_{k-1},i_k,i_{k+1},...,i_3): i_k \in [n] \}$ is called a \textit{line} of the cube C. Note that the cube $C$ has $3n^2$ lines. $C$ is said to be a \textit{Latin Cube} if for every line $l$ of $C$, \{$C(\vec{v}): \vec{v} \in l $\} = $[n]$. The set \{(i,j,k,$C(i,j,k)$) : $i,j,k \in [n]$\} is called the \textit{orthogonal array representation of $C$} and we denote it by $O(C)$.

\subsection{Basic concepts}

The identity in $S_n$ is denoted by $\epsilon$ throughout this paper. ($S_n$ is the symmetric group acting on $[n]$.) Also, the convention that permutations act from right is adopted throughout this paper. Now, we start our discussion on actions on a Latin cube. Let $\theta = (\alpha_1,\alpha_2,\alpha_3,\alpha_4) \in {S_n}^4 $ and $C$ be a Latin cube of order $n$. Now $\theta$ acts on every $(i,j,k,C(i,j,k))$ $\in O(C)$ such that $i$ is permuted by $\alpha_1$, $j$ is permuted by $\alpha_2$, $k$ is permuted by $\alpha_3$, $C(i,j,k)$ is permuted by $\alpha_4$.
Then the resulting Latin cube is denoted by $C^{\theta}$. Therefore, for each $i,j,k \in [n]$, $C^{\theta}$(i,j,k) = $C(i\alpha_1^{-1},j\alpha_2^{-1},k\alpha_3^{-1})\alpha_4$. The map $\theta$ is called an \textit{isotopism} and $C^{\theta}$ is said to be \textit{isotopic} to $C$. \\

\begin{lemma}
$\theta$ is an autotopism of a Latin cube $C$ if and only if $C(i,j,k)\alpha_4 = C(i\alpha_1,j\alpha_2,k\alpha_3)$ for each $i,j,k \in$[n].
\end{lemma}

\begin{proof}
 Suppose $\theta$ is an autotopism of a Latin cube $C$. Now $C(i,j,k) = C^{\theta}(i,j,k)$ for each $i,j,k \in [n]$. That is, $C(i,j,k) = C(i\alpha_1^{-1},j\alpha_2^{-1},k\alpha_3^{-1})\alpha_4$  for each i,j,k$\in$[n]. Without loss of generality, we replace i by $i\alpha_1$, j by $j\alpha_2$, k by $k\alpha_3$ in the above expression. Therefore, $C(i,j,k)\alpha_4 = C(i\alpha_1,j\alpha_2,k\alpha_3)$ for each i,j $\in$ [n].
Conversely suppose $C(i,j,k)\alpha_4 = C(i\alpha_1,j\alpha_2,k\alpha_3)$ for each i,j,k $\in$ [n]. Now, by replacing i by $i\alpha_1^{-1}$, j by $j\alpha_2^{-1}$ and k by $k\alpha_3^{-1}$, we have  $C(i,j,k) = C(i\alpha_1^{-1},j\alpha_2^{-1},k\alpha_3^{-1})\alpha_4$ for each i,j,k $\in$ [n]. Hence, $C = C^{\theta}$. Therefore, $\theta$ is an autotopism of $C$.
\end{proof} 

If $\theta = (\alpha,\alpha,\alpha,\alpha) \in {S_n}^4 $ with $C^{\theta} = C$ then, $\alpha$ is said to be an \textit{automorphism} of Latin cube $C$ of order n. 

Now we define $P_n$ to be the wreath product of ${{S_n}\wr{S_4}}$. Let $\sigma \in P_n$ and denote $\sigma$ as follows. $\sigma = (\alpha_1,\alpha_2,\alpha_3,\alpha_4;\delta)$ where $\alpha_1,\alpha_2,\alpha_3,\alpha_4 \in S_n$ and $\delta \in S_4$. Then, $\sigma$ acts on every $(i,j,k,C(i,j,k))$ $\in O(C)$ for a Latin cube $C$ of order n such that $i,j,k,C(i,j,k)$ are permuted by $\alpha_1,\alpha_2,\alpha_3,\alpha_4$ respectively. Then the resulting 4-tuple $(i\alpha_1,j\alpha_2,k\alpha_3,C(i,j,k)\alpha_4)$ is permuted by $\delta$. Finally resulted Latin cube, by considering permuted 4-tuples by $\delta$ as an orthogonal representation of a new Latin cube is denoted by $C^{\sigma}$. $\sigma$ is said to be a \textit{paratopism}.

\begin{example}
Suppose $(i,j,k, L_c(i,j,k))$ $\in O(C)$ for some Latin cube $C$ of order $n$. If $\delta$=$(1 4)$ then $(i,j,k,C(i,j,k))\sigma$ = $(C(i,j,k)\alpha_4,j\alpha_2,k\alpha_3,i\alpha_1)$. If $\delta$ = $(1 4 3 2) $ then $(i,j,k.C(i,j,k))\sigma$=$(j\alpha_2,k\alpha_3,C(i,j,k)\alpha_4,i\alpha_1)$.\end{example}

If $C$=$C^{\sigma}$ then $\sigma$ is said to be an \textit{autoparatopism} of Latin cube $C$. It can be observed that an isotopism is a special case of an autoparatopism with $\delta$=$\epsilon$.

In this paper, a classification of autoparatopisms of Latin cubes of order n is presented. Autoparatopisms are classified up to conjugacy of autoparatopisms, given the cycle structure of $\delta$ in an autoparatopism $\sigma$. 

\section{The Construction of Mathematical Tools}
\subsection{Cycle Structures}

It is well known that, for each $\alpha \in S_n$, $\alpha$  can be written as a product of disjoint cycles. If there is any cycle of $\alpha$ which is of length 1, it is said to be a fixed point. Set of all the fixed points of $\alpha$ is denoted by Fix$(\alpha)$. The cycle structure of $\alpha$ is denoted by $c_1^{\lambda_1}\cdot c_2^{\lambda_2}\cdot ...\cdot c_m^{\lambda_m}$ where there are $\lambda_i$ cycles of length $c_i$ and $c_1 > c_2 > ... > c_m \ge 1 $. Note that $n = \displaystyle \sum_{i=1}^m{c_i\lambda_i}$. When $\lambda_i = 1$, $c_i^1$ is denoted by just writing $c_i$. If i is a point moved by cycle $C$ in $\alpha$, then it is said that i is in $C$ and denoted by $i \in C$. Length of $C$ (size of the orbit) is denoted by $o(C)$. If $i \in C $ and $o(C)=c$ for a permutation $\alpha$, it is written that $o_{\alpha}(i)=c $.

\begin{example}
Let $\alpha = (1 2 3)(4)(5) \in S_5$. Now $c_1 = 3, \lambda_1 = 1, c_2 = 1, \lambda_2 = 2.$, Now, $2 \in (1 2 3)$ and $o((1 2 3)) = 3$, therefore $o_\alpha(2)=3$.
\end{example} 

As in \cite{mendis2017autoparatopisms}, a permutation in $S_n$ is said to be expressed in its Canonical form if, it is written as a product of disjoint cycles (Which includes 1-cycles corresponding to fixed points.) and the cycles are ordered according to their length, starting with the longest cycle and each c-cycle is of the form $(i,i+1,...,i+c-1)$, (with $i$ being referred as the leading symbol of the cycle.) and, if a cycle with leading symbol $i$ is followed by a cycle with leading symbol $j$, then $i < j$. 

\begin{example} 
$(1 2 3)(4 5)(6), (1 2 3 4 5 6), (1 2)(3 4)(5 6) \in S_6$ 
\end{example}

\subsection{Conjugates}
 $\alpha_1 \in S_n$ is said to be conjugate to $\alpha_2 \in S_n$ if there is a $\beta \in S_n$ such that $\alpha_2 = \beta^{-1}\alpha_1\beta$. We denote $\alpha_1$ being conjugate to $\alpha_2$ by $\alpha_1 \sim \alpha_2$. Also $\sigma_1 \in S_n \wr S_4$ is said to be conjugate to $\sigma_2 \in S_n \wr S_4$ if there is a $\tau \in S_n \wr S_4$ such that $\sigma_2 = \tau^{-1}\sigma_1\tau$. Similarly the notation $\sigma_1 \sim \sigma_2$ is used to denote $\sigma_1$ being conjugate to $\sigma_2$. Note that "$\sim$" defines an equivalence relation on $S_n$ (or $S_n \wr S_4$).

\begin{lemma}
$\alpha_1 \in S_n$ and $\alpha_2 \in S_n$ are conjugate if and only if they are of the same cycle structure.
\end{lemma}  
\begin{proof} Suppose $\alpha_1$ and $\alpha_2$ are conjugate in $S_n$. Then there is $\beta \in S_n$ such that $\alpha_2 = \beta^{-1}\alpha_1\beta$. Let $i \in [n]$. Then there is $j \in [n]$ such that $(i)\alpha_1 = j$. Now, $((i)\beta)\alpha_2 = ((i)\beta)\beta^{-1}\alpha_1\beta = ((i)\alpha_1)\beta = (j)\beta$. Thus, if $\alpha_1$ maps $i$ to $j$ then $\alpha_2$ maps $(i)\beta$ to $(j)\beta$. Hence $\alpha_1$ and $\alpha_2$ have the same cycle structure. Conversely suppose both $\alpha_1$ and $\alpha_2$ have the same cycle structure. Now we construct a $\beta \in S_n$ such that for every cycle $(a_1 a_2 ... a_k)$ in $\alpha_1$, we choose a cycle $(b_1 b_2 ... b_k)$ of length k from $\alpha_2$ and map each $a_i$ to $b_i$. Then $\beta^{-1}\alpha_1\beta = \alpha_2$. Hence $\alpha_1 \sim \alpha_2$. 
\end{proof}

The \textit{Hamming distance} dist$(C,C')$ between two Latin cubes $C$ and $C'$ of same order is defined as the cardinality of the set $\{d \in O(C) : d \not\in O(C')\}$. Note that If $\sigma \in S_n \wr S_4 $ then, $\sigma$ is an autoparatopism if and only if dist$(C,C^{\sigma}) = 0$. 

Now we extend \textbf{theorem 2.1} and \textbf{theorem 2.2} of \cite{mendis2016latin} to \textbf{Lemma 2.2:} and \textbf{Theorem 2.2} respectively and appropriately as follows.

\begin{lemma} Suppose $\sigma_1$ and $\sigma_2$ are conjugate in $S_n \wr S_4$. Then, $\sigma_1$ is an autoparatopism if and only if $\sigma_2$ is an autoparatopism. \end{lemma}

\begin{proof} Since $\sigma_1$ and $\sigma_2$ are conjugate in $S_n \wr S_4$, there is $\tau \in S_n \wr S_4$ such that $\sigma_2 = \tau^{-1}\sigma_1\tau$. Suppose $\sigma_1 $ is an autoparatopism of a Latin cube C. Therefore, $C = C^{\sigma_1}$. Subsequently we have $C^{\tau} = C^{\sigma_1\tau}$ by applying same action on the both sides of the equation. Hence, dist($C^{\tau},C^{\sigma_1\tau})=0$ .In order to show that, $\sigma_2$ is an autoparatopism of the Latin cube $C^{\tau}$ we compute the hamming distance between $C^{\tau}$ and $C^{\tau\sigma_2}$. That is, dist($C^{\tau},C^{\tau\sigma_2})$ = dist($C^{\tau},C^{\tau(\tau^{-1}\sigma_1\tau)})$ = dist($C^{\tau},C^{\sigma_1\tau}) = 0$. Therefore we have the required result. Proving the converse is about just interchanging the roles of $\sigma_1$ and $\sigma_2$. \end{proof}

\begin{theorem} Let $\sigma_1$ = $(\alpha_1,\alpha_2,\alpha_3,\alpha_4;\delta_1) \in S_n\wr S_4$ and $\sigma_2$ = $(\beta_1,\beta_2,\beta_3,\beta_4;\delta_2) \in S_n \wr S_4$. Then, $\sigma_1$ is conjugate to $\sigma_2$ in $S_n\wr S_4$ if and only if there is a length preserving bijection $\eta$ from the cycles of $\delta_1$ to the cycles of $\delta_2$ such that if $\eta$ maps a cycle $(a_1...a_k)$ to $(b_1...b_k)$ then $\alpha_{a_1}\alpha_{a_2}...\alpha_{a_k} \sim \beta_{b_1}\beta_{b_2}...\beta_{b_k}$. \end{theorem}

\begin{proof} Suppose $\sigma_1$ is conjugate to $\sigma_2$ in $S_n\wr S_4$. Then there is $\tau = (\gamma_1,\gamma_2,\gamma_3,\gamma_4;\delta_3) \in S_n\wr S_4$ such that $\sigma_2 = \tau^{-1}\sigma_1\tau$.

\underline{Case I:} $\delta_3 = \epsilon$.
If $\delta_1 = \epsilon $ then $\sigma_2 = \tau^{-1}\sigma_1\tau$ = $(\gamma_1^{-1},\gamma_2^{-1},\gamma_3^{-1},\gamma_4^{-1};\epsilon)(\alpha_1,\alpha_2,\alpha_3,\alpha_4;\epsilon)(\gamma_1,\gamma_2,\gamma_3,\gamma_4;\epsilon)$. That is $(\beta_1,\beta_2,\beta_3,\beta_4;\delta_2) = (\gamma_1^{-1}\alpha_1\gamma_1,\gamma_2^{-1}\alpha_2\gamma_2,\gamma_3^{-1}\alpha_3\gamma_3,\gamma_4^{-1}\alpha_4\gamma_4;\epsilon)$. Therefore $\delta_2 = \epsilon$. Hence we take $\eta$ to be the map which assigns each 1-cycle to itself. Now $\eta$ preserves the length of cycles of $\delta_1$. Also we observe $\beta_i = \gamma_i^{-1}\alpha_i\gamma_i $ for $1 \le i \le 4$. Thus, $\beta_i \sim \alpha_i$ for $1 \le i \le 4$. If $\delta_1 \ne \epsilon $ then it has non-fixed points. $\sigma_2 = \tau^{-1}\sigma_1\tau$ = $(\gamma_1^{-1},\gamma_2^{-1},\gamma_3^{-1},\gamma_4^{-1};\epsilon)(\alpha_1,\alpha_2,\alpha_3,\alpha_4;\delta_1)(\gamma_1,\gamma_2,\gamma_3,\gamma_4;\epsilon)$. That is $(\beta_1,\beta_2,\beta_3,\beta_4;\delta_2) = (\gamma_1^{-1}\alpha_1\gamma_{1\delta_1 },\gamma_2^{-1}\alpha_2\gamma_{2\delta_1 },\gamma_3^{-1}\alpha_3\gamma_{3\delta_1 },\gamma_4^{-1}\alpha_4\gamma_{4\delta_1};\delta_1)$. Therefore $\delta_2 = \delta_1$. Hence we take $\eta$ to be the map which assigns each cycle to itself. Now $\eta$ preserves the length of cycles of $\delta_1$. Also we observe $\beta_i = \gamma_i^{-1}\alpha_i\gamma_{i\delta_1}$ for $1 \le i \le 4$. If $i$ is a fixed point of $\delta$ then $\beta_i = \gamma_i^{-1}\alpha_i\gamma_i$. Hence, $\beta_i \sim \alpha_i$. If $i$ is a non-fixed point, then there is cycle $(a_1 a_2 ... a_k)$ of $\delta_1$ such that $i \in (a_1 a_2 ... a_k)$. Thus, $\beta_{a_j} = \gamma_{a_j}\alpha_{a_j}\gamma_{a_{j+1}}^{-1}$ for $1 \le j \le k-1$. When $j = k$ we have $\beta_{a_k} = \gamma_{a_k}\alpha_{a_k}\gamma_{a_1}^{-1}$. Therefore $\beta_{b_1}\beta_{b_2}...\beta_{b_k} = \gamma_{a_1}^{-1}\alpha_{a_1}\alpha_{a_2}...\alpha_{a_k}\gamma_{a_1}$. Thus, $\alpha_{a_1}\alpha_{a_2}...\alpha_{a_k} \sim \beta_{b_1}\beta_{b_2}...\beta_{b_k}$.

\underline{Case II:} $\delta_3 \ne \epsilon$.
Now we write $\tau = \tau_1\tau_2$ where $\tau_1 = (\gamma_1,\gamma_2,\gamma_3,\gamma_4;\epsilon)$ and $\tau_2 = (\epsilon,\epsilon,\epsilon,\epsilon;\delta_3)$. Now observe that, $\sigma_2 = \tau^{-1}\sigma_1\tau = (\tau_1\tau_2)^{-1}\sigma_1(\tau_1\tau_2) = \tau_2^{-1}(\tau_1^{-1}\sigma_1\tau_1)\tau_2$. Therefore, $\tau_1^{-1}\sigma_1\tau_1$ is conjugated by $\tau_2$. Hence say $\tau_1^{-1}\sigma_1\tau_1 = (\mu_1,\mu_2,\mu_3,\mu;\delta_1)$. Now, $(\mu_1,\mu_2,\mu_3.\mu;\delta_1)$ can be found by applying \underline{Case I} of this proof. Now, we obtain $\sigma_2 = (\epsilon,\epsilon,\epsilon\epsilon;\delta_3^{-1})(\mu_1,\mu_2,\mu_3,\mu;\delta_1)(\epsilon,\epsilon,\epsilon,\epsilon;\delta_3)$ = $(\mu_{1\delta_3^{-1}},\mu_{2\delta_3^{-1}},\mu_{3}\delta_3^{-1},\mu_{4\delta_3^{-1}};\delta_3^{-1}\delta_1\delta_3)$. Since $\delta_2 = \delta_3^{-1}\delta_1\delta_3$ we take $\eta$ to be the map which assigns each cycle $(a_1,a_2,...,a_3)$ to $((a_1\delta_3)(a_2\delta_3),...(a_k\delta_3))$. Hence it is clear that $\eta$ preserves the lengths. It can be observed that $\eta$ also agrees the condition  $\alpha_{a_1}\alpha_{a_2}...\alpha_{a_k} \sim \beta_{b_1}\beta_{b_2}...\beta_{b_k}$.

Conversely suppose the existence of $\eta$ satisfying the given condition. We know that $S_4$ has permutations of cycle structures $1^4,3.1,2.1^2,2^2$ and $4$. Given a cycle structure, all the permutations in $S_4$ with that cycle structure results similar proofs. Hence we only prove the results for the cases $\delta_1 = (1 2), (1 2 3), (1 2 3 4), (1 2)(3 4)$. ($\delta_1 = \epsilon$ case is trivial.) Also we assume $\eta$ to be the identity to omit the repetition of similar cases.

If $\delta_1 = (1 2)$ then, $\alpha_1\alpha_2 \sim \beta_1 \beta_2$, $\alpha_3 \sim \beta_3$ and $\alpha_4 \sim \beta_4$. Thus, there exist $r,s,t \in S_n$ such that $r^{-1}\alpha_1\alpha_2r = \beta_1 \beta_2$, $s^{-1}\alpha_3s = \beta_3$ and $t^{-1}\alpha_4t = \beta_4$. Choose $\tau \in S_n$ to be $(r,\alpha_1^{-1}r\beta_1,s,t;\epsilon)$. Then $\tau^{-1}\sigma_1\tau$ = $(\beta_1,\beta_2,\beta_3,\beta_4;(1 2))$ = $\sigma_2$. If $\delta_1 = (1 2 3)$ then, $\alpha_1\alpha_2\alpha_3 \sim \beta_1 \beta_2\beta_3$ and $\alpha_4 \sim \beta_4$. Thus, there exist $r,s \in S_n$ such that $r^{-1}\alpha_1\alpha_2\alpha_3r = \beta_1 \beta_2\beta_3$ and $s^{-1}\alpha_4s = \beta_4$. Choose $\tau \in S_n$ to be $(r,\alpha_1^{-1}r\beta_1,\alpha_3r\beta_3^{-1},s;\epsilon)$. Then $\tau^{-1}\sigma_1\tau$ = $(\beta_1,\beta_2,\beta_3,\beta_4;(1 2 3))$ = $\sigma_2$. If $\delta_1 = (1 2 3 4)$ then, $\alpha_1\alpha_2\alpha_3\alpha_4 \sim \beta_1 \beta_2\beta_3\beta_4$ . Thus, there exists $r \in S_n$ such that $r^{-1}\alpha_1\alpha_2\alpha_3\alpha_4r = \beta_1 \beta_2\beta_3\beta_4$. Choose $\tau \in S_n$ to be $(r,\alpha_1^{-1}r\beta_1,\alpha_2^{-1}\alpha_1^{-1}r\beta_1\beta_2,\alpha_4r\beta_4^{-1};\epsilon)$. Then $\tau^{-1}\sigma_1\tau$ = $(\beta_1,\beta_2,\beta_3,\beta_4;(1 2 3 4))$ = $\sigma_2$. If $\delta_1 = (1 2)(3 4)$ then, $\alpha_1\alpha_2 \sim \beta_1 \beta_2$ and $\alpha_3\alpha_4 \sim \beta_3\beta_4$. Thus, there exist $r,s \in S_n$ such that $r^{-1}\alpha_1\alpha_2r = \beta_1 \beta_2$ and $s^{-1}\alpha_3\alpha_4s = \beta_4$. Choose $\tau \in S_n$ to be $(r,\alpha_1^{-1}r\beta_1,s,\alpha_3^{-1}s\beta_3;\epsilon)$. Hence we obtain
$\tau^{-1}\sigma_1\tau$ = $(\beta_1,\beta_2,\beta_3,\beta_4;(1 2)(3 4))$ = $\sigma_2$. Therefore in either case, $\sigma_1 \sim \sigma_2$ in $S_n \wr S_4$. \end{proof}

\section{The Classification}

Suppose $\sigma_1 = (\alpha_1,\alpha_2,\alpha_3,\alpha_4;\delta_1)$ is an autoparatopism of Latin cube C. Then by theorem 2.3 we are able to confirm the cases for $\sigma_2$ in the following table are conjugate to $\sigma_1$. Hence those become autoparatopisms of the Latin cube $C^{\tau}$ by Lemma 2.2, where $\tau \in S_n \wr S_4$ is the appropriate conjugating element in each case.

\begin{center}
   \begin{tabular}{|c|c|c|}
   \hline
   $\delta_1$  & $\eta$ & $\sigma_2$ \\
   \hline \hline
   $\epsilon$  &identity map& $(\alpha_1,\alpha_2,\alpha_3,\alpha_4;\epsilon)$ \\
   \hline
   $(1 2)$&identity map& $(\epsilon,\alpha_1\alpha_2,\alpha_3,\alpha_4;(1 2))$ \\
   $(1 3)$&$(1 3)\rightarrow(1 2)$,$(2)\rightarrow(3)$, $(4)\rightarrow(4)$
   & $(\epsilon,\alpha_1\alpha_3,\alpha_2,\alpha_4;(1 2))$ \\
   $(1 4)$&$(1 4)\rightarrow(1 2)$,$(2)\rightarrow(3)$, $(3)\rightarrow(4)$
   & $(\epsilon,\alpha_1\alpha_4,\alpha_2,\alpha_3;(1 2))$ \\
   $(2 3)$&$(2 3)\rightarrow(1 2)$,$(1)\rightarrow(3)$, $(4)\rightarrow(4)$
   & $(\epsilon,\alpha_2\alpha_3,\alpha_1,\alpha_4;(1 2))$ \\
   $(2 4)$&$(2 4)\rightarrow(1 2)$,$(1)\rightarrow(3)$, $(3)\rightarrow(4)$
   & $(\epsilon,\alpha_2\alpha_4,\alpha_1,\alpha_3;(1 2))$ \\
   $(3 4)$&$(3 4)\rightarrow(1 2)$,$(1)\rightarrow(3)$, $(2)\rightarrow(4)$
   & $(\epsilon,\alpha_3\alpha_4,\alpha_1,\alpha_2;(1 2))$ \\
  \hline
   $(1 2 3)$&identity map& $(\epsilon,\epsilon,\alpha_1\alpha_2\alpha_3,\alpha_4;(1 2 3))$ \\
   $(1 3 2)$&$(1 3 2)\rightarrow(1 2 3),(4)\rightarrow(4)$& $(\epsilon,\epsilon,\alpha_1\alpha_3\alpha_2,\alpha_4;(1 2 3))$ \\
   $(1 2 4)$&$(1 2 4)\rightarrow(1 2 3),(3)\rightarrow(4)$& $(\epsilon,\epsilon,\alpha_1\alpha_2\alpha_4,\alpha_3;(1 2 3))$ \\
   $(1 4 2)$&$(1 4 2)\rightarrow(1 2 3),(3)\rightarrow(4)$& $(\epsilon,\epsilon,\alpha_1\alpha_4\alpha_2,\alpha_3;(1 2 3))$ \\
   $(1 3 4)$&$(1 3 4)\rightarrow(1 2 3),(2)\rightarrow(4)$& $(\epsilon,\epsilon,\alpha_1\alpha_3\alpha_4,\alpha_2;(1 2 3))$ \\
   $(1 4 3)$&$(1 4 3)\rightarrow(1 2 3),(2)\rightarrow(4)$& $(\epsilon,\epsilon,\alpha_1\alpha_4\alpha_3,\alpha_2;(1 2 3))$ \\
   $(2 3 4)$&$(2 3 4)\rightarrow(1 2 3),(1)\rightarrow(4)$& $(\epsilon,\epsilon,\alpha_2\alpha_3\alpha_4,\alpha_1;(1 2 3))$ \\
   $(2 4 3)$&$(2 3 4)\rightarrow(1 2 3),(1)\rightarrow(4)$& $(\epsilon,\epsilon,\alpha_2\alpha_4\alpha_3,\alpha_1;(1 2 3))$ \\
  \hline
   $(1 2 3 4)$&identity map& $(\epsilon,\epsilon,\epsilon,\alpha_1\alpha_2\alpha_3\alpha_4;(1 2 3 4))$ \\
   $(1 2 4 3)$&$(1 2 4 3)\rightarrow(1 2 3 4)$& $(\epsilon,\epsilon,\epsilon,\alpha_1\alpha_2\alpha_4\alpha_3;(1 2 3 4))$ \\
   $(1 3 2 4)$&$(1 3 2 4)\rightarrow(1 2 3 4)$& $(\epsilon,\epsilon,\epsilon,\alpha_1\alpha_3\alpha_2\alpha_4;(1 2 3 4))$ \\
   $(1 3 4 2)$&$(1 3 4 2)\rightarrow(1 2 3 4)$& $(\epsilon,\epsilon,\epsilon,\alpha_1\alpha_3\alpha_4\alpha_2;(1 2 3 4))$ \\
   $(1 4 3 2)$&$(1 4 3 2)\rightarrow(1 2 3 4)$& $(\epsilon,\epsilon,\epsilon,\alpha_1\alpha_4\alpha_3\alpha_2;(1 2 3 4))$ \\
   $(1 4 2 3)$&$(1 4 2 3)\rightarrow(1 2 3 4)$& $(\epsilon,\epsilon,\epsilon,\alpha_1\alpha_4\alpha_2\alpha_3;(1 2 3 4))$ \\
   \hline
   $(1 3)(2 4)$&identity map& $(\epsilon,\epsilon,\alpha_1\alpha_3,\alpha_2\alpha_4;(1 3)(2 4))$ \\
   $(1 2)(3 4)$&$(1 2)\rightarrow(1 3),(3 4)\rightarrow(2 4)$& $(\epsilon,\epsilon,\alpha_1\alpha_2,\alpha_3\alpha_4;(1 3)(2 4))$ \\
   $(1 4)(2 3)$&$(1 4)\rightarrow(1 3),(2 3)\rightarrow(2 4)$& $(\epsilon,\epsilon,\alpha_1\alpha_4,\alpha_2\alpha_3;(1 3)(2 4))$ \\
   \hline
   \end{tabular}
\end{center}

Therefore, it is sufficient to determine the autoparatopisms of the forms $(\beta_1,\beta_2,\beta_3,\beta_4;\epsilon)$, $(\epsilon,\beta_2,\beta_3,\beta_4;(1 2))$, $(\epsilon,\epsilon,\beta_3,\beta_4;(1 2 3))$,$(\epsilon,\epsilon,\epsilon,\beta_4;(1 2 3 4))$ and $(\epsilon,\epsilon,\beta_3,\beta_4;(1 3)(2 4))$. Then by conjugating autoparatopisms in the each form by every element in the group $S_n \wr S_4$, we are able determine all the autoparatopisms. 

\section{Concluding Remarks}

The whole discussion of classifying autoparatopisms on Latin cubes depends on conjugates of autoparatopisms being conjugates and cycle structure of the $\delta$ in the element $\sigma = (\alpha_1,\alpha_2,\alpha_3,\alpha_4;\delta)$ of $S_n \wr S_4$. These concepts can be carefully extended to a hypercube of any dimension d. The objective of determining the cycle structures of autoparatopisms given a hypercube of order n and dimension d is still an open problem.

\section{Acknowledgements}

I would like to express my gratitude towards Dr. Mrs. M.J.L. Mendis for introducing me the concept of the autoparatopims of Latin squares, which led me to generalize the concept for cubes.

\bibliographystyle{unsrt}  
\bibliography{references} 

\end{document}